\documentclass[Threeside,11pt]{article}
\usepackage{amsmath, amsthm, amssymb, bm, graphicx, hyperref, mathrsfs,indentfirst,abstract}
\allowdisplaybreaks[4]

\usepackage[numbers,sort&compress]{natbib}
\allowdisplaybreaks

\numberwithin{equation}{section}

\newtheorem{theorem}{Theorem}[section]
\newtheorem{lemma}{Lemma}[section]
\newtheorem{definition}{Definition}[section]
\newtheorem{proposition}{Proposition}[section]

\begin{document}
	\title{{\sl The rotating periodic, spiral-like almost periodic and spiral-like almost automorphic solutions of Navier-Stokes equations with the Coriolis force}}
	
	\author{Ziying Chen $^{\mathcal{z},1}$, Yong Li $^{*,\mathcal{x},1,2}$ }

\renewcommand{\thefootnote}{}
\footnotetext{\hspace*{-6mm}
	
	\begin{tabular}{l l}
		$^{*}$~~~The corresponding author.\\
		$^\mathcal{z}$~~~E-mail address : zychen21@mails.jlu.edu.cn\\
		$^{\mathcal{x}}$~~~E-mail address : liyong@jlu.edu.cn\\
		$^{1}$~~~School of Mathematics, Jilin University, Changchun 130012, People's Republic of China.\\
		$^{2}$~~~School of Mathematics and Statistics, Center for Mathematics and Interdisciplinary Sciences,\\
		~~~ Northeast Normal University, Changchun 130024, People's Republic of China.
\end{tabular}}
\date{}
\maketitle
\begin{abstract}
    We consider the spatio-temporal periodic problem for the Navier-Stokes equations with a small external force in the rotational framework. We prove the existence and uniqueness of the rotating periodic, spiral-like almost periodic and spiral-like almost automorphic solutions of Navier-Stokes equations with the Coriolis force.\\
	\par\textbf{Keywords:} Navier-Stokes equations; Coriolis force; affine-periodic solution; spiral-like almost periodic solution; spiral-like almost automorphic solution 

\end{abstract}

   

	\section{Introduction}
	
	  We study the following Navier-Stokes equation with the Coriolis force:
	\begin{equation}
		\begin{cases}\label{1.1}
			\frac{\partial u}{\partial t} = \Delta u - le_{3}\times u - \left ( u\cdot \nabla  \right ) u - \nabla \pi  + f, \qquad &in\quad \mathbb{R} \times \mathbb{R} ^{3}, \\
			\nabla \cdot u=0,&in \quad \mathbb{R} \times \mathbb{R} ^{3},
		\end{cases}
	\end{equation}
where $ t\in \mathbb{R} , x\in \mathbb{R} ^{3} , u = u\left( t,x\right) = \left( u^{1}\left( t,x\right) ,u^{2}\left( t,x\right),u^{3}\left( t,x\right) \right)  $ denotes the velocity at $ \left( t,x\right)  $, $ \pi=\pi\left( t,x\right)  $ denotes the pressure at $ \left( t,x\right)  $, $ f = f\left( t,x\right) = \left( f^{1}\left( t,x\right) ,f^{2}\left( t,x\right),f^{3}\left( t,x\right) \right) $ denotes the external force. Here, $ l\in \mathbb{R}  $ represents the speed of rotation around the vertical unit vector $ e_{3}=\left( 0,0,1\right)  $, which is called the Coriolis parameter. 

It is well known that problems related to large-scale atmospheric and oceanic flows are dominated by rotational effects. The Coriolis force appears in many oceanographic and meteorological models. Therefore, it is of great importance to study the Navier-Stokes equations with Coriolis force. For the initial value problems of $ \eqref{1.1} $, Babin et al. \cite{ref1,ref2,ref3} and Chemin et al. \cite{ref4} proved the global existence and regularity of solutions to equation $ \eqref{1.1} $. The local existence of solutions to the Navier-Stokes equations with the Coriolis force can be found in \cite{ref24,ref23,ref25}. The global existence of solutions to the rotating Navier-Stokes equations has been studied by several authors. Giga et al. \cite{ref8,ref9} proved the existence of global solutions to the rotating Navier-Stokes equations for a class of spatially nondecaying initial data which includes a large class of almost periodic funtions. Yoneda \cite{ref13} obtained the existence of solutions to the rotating Navier–Stokes equations with spatially almost periodic large data provided by a sufficiently large Coriolis force for large times. For more studies of Navier-Stokes equations, see \cite{ref5,ref26,ref10,ref6} and the references therein.

Periodicity is a crucial property of the physical world. The periodic solution problems for various physical models have been extensively studied. Cheng and Li \cite{ref27} proved the existence and the uniqueness of time periodic solutions for a full hydrodynamic semiconductor model. Cai and Tan \cite{ref28} established the existence of weak time periodic solutions for the compressible Navier–Stokes–Poisson equations. And the time periodic problem of the Navier-Stokes equations in the rotational framework has been studied. Iwabuchi and Takada \cite{ref12} proved the existence and the uniqueness of a time periodic mild solution to $ \eqref{1.1} $ for the large time periodic external force $ f $ when the speed of rotation is sufficiently fast. Kozono, Mashiko and Takada \cite{ref15} showed that if the time periodic force $ f $ is uniformly small with respect to the Coriolis parameter, then there exists a unique time periodic mild solutions to $ \eqref{1.1} $ with the same period as $ f $. Koh, Lee and Takada \cite{ref11} proved the existence and uniqueness of time periodic solution to the rotating  Navier-Stokes equations for large initial velocity and large external force provided that the speed of rotation is sufficiently high. 

However, it is known that not all natural phenomena can be described alone by periodicity. Recently, ones introduced and studied affine-periodic solutions, see \cite{ref17}. Such solutions describe some kinds of spatio-temporal symmetric motions like dynamics of rigid body, spiral waves, Bloch waves and so on. Li and Huang \cite{ref18} discussed Levinson’s problem on affine-periodic solutions where the system is dissipative-repulsive. Zhang et al. \cite{ref22} obtained the existence of $ \left( Q,T\right) $-periodic patterns of reaction diffusion systems. For more research, see \cite{ref21,ref19,ref20}.

Although some recurrent properties of solutions of Navier-Stokes equations have been studied in \cite{ref16}, the affine-periodic solutions of Navier-Stokes equations with the Coriolis force have not been studied. If system $ \eqref{1.1} $ is a $ \left( Q,T\right) $-affine-periodic system, i.e. $ f\left( t+T,x \right) = Qf\left( t,Q^{-1}x \right), \pi \left ( t+T,x \right ) = \pi \left ( t,Q^{-1}x \right ) $, where
\begin{displaymath}
	Q = \begin{pmatrix}
		e^{\alpha }\cos \theta  & -e^{\alpha }\sin \theta  & 0 \\
		e^{\alpha }\sin \theta  & e^{\alpha }\cos \theta  & 0 \\
		0 & 0 & 1
	\end{pmatrix},\alpha \in \mathbb{R},\theta \in \left [ 0,2\pi \right ),T>0,
\end{displaymath}
let $ v\left( t,x\right) = Q^{-1}u\left( t+T,Qx\right) $. By coordinate transformation we obtain the following equation:
\begin{equation}
	\begin{cases}\label{1.2}
		\frac{\partial v}{\partial t} =  \bar{\Delta} v - le_{3}\times v - \left ( v\cdot \nabla  \right ) v - M\nabla \pi + f , \\
		\nabla \cdot v=0,
	\end{cases}
\end{equation}
where $ \bar{\Delta} = e^{-2\alpha}\frac{\partial^2}{\partial x_1^{2}} + e^{-2\alpha} \frac{\partial^2}{\partial x_2^{2}} + \frac{\partial^2}{\partial x_3^{2}} $, and $ M = \begin{pmatrix}
	e^{-2\alpha } & 0 & 0 \\
	0 & e^{-2\alpha } & 0 \\
	0 & 0 & 1
\end{pmatrix} $.\\

 After that we will study the operator semigroup generated by equation $ \eqref{1.2} $. In particular, when $ \alpha = 0 $, equation $ \eqref{1.2} $ is the original equation $ \eqref{1.1} $. In this paper, we will introduce the new concepts with spatio-temporal structure, such as spiral-like almost periodic solutions and spiral-like almost automorphic solutions. Hence, we generalize the time-periodic problem of the Navier-Stokes equations to spatio-temporal patterns. The primary objective of this paper is to establish the existence and uniqueness of rotating periodic, spiral-like almost periodic and spiral-like almost automorphic solutions of equation $ \eqref{1.1} $. Our result shows that if the external force $ f $ has a rotating spatio-temporal structure, then equation $ \eqref{1.1} $ also has a solution evolving in the same structure.

The rest of the article is organized as follows. In Section 2, we introduce some preliminaries about function spaces and $ L^{p}-L^{q} $ estimates of the corresponding operator semigroup of equation $ \eqref{1.2} $. In addition, the definitions of affine-periodic, spiral-like almost periodic and spiral-like almost automorphic solutions are introduced. In Section 3, we prove the existence and uniqueness of the rotating periodic, spiral-like almost periodic and spiral-like almost automorphic solutions of equation $ \eqref{1.1} $ when $ \alpha =0 $.

\section{Preliminaries}

Throughout this paper, we shall denote by $ C $ the constants which may change from line to line. We shall introduce some notations and function spaces. Let $ C_{0,\sigma }^{\infty }\left ( \mathbb{R}^{3}  \right )  $ denote the set of all real-valued $ C^{\infty } $ vector functions $ \phi = \left ( \phi ^{1}, \phi ^{2},\phi ^{3}\right )  $ with compact support in $ \mathbb{R}^{3} $ such that $ div\phi = 0 $. For $ 1<r<\infty $, $ L_{\sigma }^{r}\left ( \mathbb{R}^{3}  \right ) $ is defined as the closure of $ C_{0,\sigma }^{\infty }\left ( \mathbb{R}^{3}  \right ) $ with respect to the $ L^{r} $-norm $ \left \| \cdot  \right \| _{r} $. $ \left ( \cdot,\cdot   \right )  $ denotes the duality pairing between $ L_{\sigma }^{r}\left ( \mathbb{R}^{3} \right ) $ and $ L_{\sigma }^{r^{\prime}}\left ( \mathbb{R}^{3}  \right ) $, where $ 1/r+1/r^{\prime} = 1 $. $ L^{r}\left ( \mathbb{R}^{3}  \right ) $ and $ \dot{W}^{1,r}\left ( \mathbb{R}^{3}  \right ) $ stand for the usual vector-valued Lebesgue space and homogeneous Sobolev space over $ \mathbb{R}^{3} $, respectively. Let $ \mathscr{S}\left( \mathbb{R}^{3} \right ) $ be the Schwartz class of all rapidly decreasing functions, and let $ \mathscr{S}' \left( \mathbb{R}^{3} \right )$ be the space of all tempered distributions. We first recall the definition of the homogeneous Littlewood-Paley decompositions. Let  $ \varphi $ be a radial function in $ \mathscr{S} \left( \mathbb{R}^{3} \right ) $ satisfying the following properties:
\begin{displaymath}
	0\le \hat{\varphi}\left( \xi\right) \le 1 \quad \text{for all $ \xi \in \mathbb{R}^{3} $}, 		
\end{displaymath}
\begin{displaymath}
	supp\hat{\varphi } \subset \left \{ \xi \in \mathbb{R}^{3} \mid 2^{-1}\le \left | \xi  \right |\le 2 \right \},
\end{displaymath}
and
\begin{displaymath}
	\sum \limits _{j\in \mathbb{Z} }\hat{\varphi }_{j}\left ( \xi  \right ) = 1 \quad \text{for all $ \xi \in \mathbb{R}^{3} \setminus \left \{ 0 \right \}  $},
\end{displaymath}
where $ \varphi _{j}\left ( x \right ) := 2^{3j}\varphi \left ( 2^{j}x \right ) $. Then, for $ s\in \mathbb{R} $ and $ 1 \le p,q \le \infty $, the homogeneous Besov space $ \dot{B}_{p,q} ^{s}\left( \mathbb{R}^{3}\right) $ is defined as the set of all tempered distributions $ f \in \mathscr{S}' \left( \mathbb{R}^{3} \right ) $ such that the following semi-norm is finite:
\begin{displaymath}
	\left \| f \right \| _{\dot{B}_{p,q} ^{s}} := \left \| \left \{ 2^{sj}\left \| \varphi _{j}\ast f \right \|_{p}  \right \} _{j \in \mathbb{Z} } \right \| _{q}.
\end{displaymath}

In the following we consider semigroup generated by equation $ \eqref{1.2} $. Similar to the derivation in \cite{ref7}.
First we consider the linear equation
\begin{equation}\label{2.1}
	\begin{split}
		v_{t} - \bar{\Delta}v + le_{3}\times v + M\nabla \pi = 0,\quad x\in \mathbb{R}^{3},t>0,\\
		\nabla\cdot v = 0,\quad x\in \mathbb{R}^{3},t>0,\\
		v\left( 0,x\right) = v_{0}\left( x\right),\quad x\in \mathbb{R}^{3},
	\end{split} 
\end{equation}
and the corresponding resolvent equation in $ L^{p} $ spaces. More precisely, let $ \lambda \in \sum_{\phi} $ for some $ \phi \in \left[ 0,\pi/2 \right) $, where $ \sum_{\phi} =  \left \{ z\in \mathbb{C} \setminus \left \{ 0 \right \}  ,\left | \mathrm{arg} z \right |< \phi \right \} $ and let $ g \in L^{p}_{\sigma}\left( \mathbb{R}^{3}\right) $. Consider
 \begin{equation}\label{2.2}
 	\begin{split}
 		\lambda v - \bar{\Delta}v + le_{3}\times v + M\nabla \pi = g,\quad x\in \mathbb{R}^{3},t>0,\\
 		\nabla\cdot v = 0,\quad x\in \mathbb{R}^{3},t>0.
 	\end{split} 	
 \end{equation}
Taking the Fourier transform of $ \eqref{2.2} $, we have
\begin{equation}\label{2.3}
	\begin{split}
		\left( \lambda + e^{-2\alpha} \xi_{1}^{2} + e^{-2\alpha} \xi_{2}^{2} + \xi_{3}^{2} \right) \widehat{v_{1}} - l\widehat{v_{2}} + e^{-2\alpha}i\xi_{1}\widehat{\pi} = \widehat{g_{1}}\\
		\left( \lambda + e^{-2\alpha} \xi_{1}^{2} + e^{-2\alpha} \xi_{2}^{2} + \xi_{3}^{2} \right) \widehat{v_{2}} + l\widehat{v_{1}} + e^{-2\alpha}i\xi_{2}\widehat{\pi} = \widehat{g_{2}}\\
		\left( \lambda + e^{-2\alpha} \xi_{1}^{2} + e^{-2\alpha} \xi_{2}^{2} + \xi_{3}^{2} \right) \widehat{v_{3}}  + i\xi_{3}\widehat{\pi} = \widehat{g_{3}}.
	\end{split}
\end{equation}
Note that $ \nabla\cdot g = 0 $, $ \widehat{\nabla \cdot g} = i\xi_{1} \widehat{g_{1}} + i\xi_{2} \widehat{g_{2}} + i\xi_{3} \widehat{g_{3}} = 0 $. This yields
\begin{displaymath}
	-e^{-2\alpha} \xi_{1}^{2}\widehat{\pi} \left( \xi\right) - e^{-2\alpha} \xi_{2}^{2}\widehat{\pi} \left( \xi\right) - \xi_{3}^{2}\widehat{\pi} \left( \xi\right) + l\left( -i\xi_{1}\widehat{v_{2}}\left( \xi\right) + i\xi_{2}\widehat{v_{1}}\left( \xi\right) \right) = 0,
\end{displaymath}
and hence
\begin{equation}\label{2.4}
	\widehat{\pi}\left( \xi\right) = \frac{l}{\left | \bar{\xi} \right | ^{2}} \left[ i\xi_{2}\widehat{v_{1}}\left( \xi\right) - i\xi_{1}\widehat{v_{2}}\left( \xi\right) \right],
\end{equation}
where 
\begin{displaymath}
		\bar{\xi} = \left( e^{-\alpha}\xi_{1},e^{-\alpha}\xi_{2},\xi_{3} \right),
\end{displaymath} 
\begin{displaymath}
		\left | \bar{\xi} \right | ^{2} = e^{-2\alpha} \xi_{1}^{2} + e^{-2\alpha} \xi_{2}^{2} + \xi_{3}^{2}.
\end{displaymath} 
 From $ \eqref{2.3} $ and $ \eqref{2.4} $, we obtain 
\begin{displaymath}
	\begin{pmatrix}
	\lambda + \left | \bar{\xi} \right | ^{2} - e^{-2\alpha}l\frac{\xi_{1}\xi_{2}}{\left | \bar{\xi} \right | ^{2}} 	& -l + e^{-2\alpha}l\frac{\xi_{1}^{2}}{\left | \bar{\xi} \right | ^{2}}  & 0 \\
	l - e^{-2\alpha}l\frac{\xi_{2}^{2}}{\left | \bar{\xi} \right | ^{2}} 	& \lambda + \left | \bar{\xi} \right | ^{2} + e^{-2\alpha}l\frac{\xi_{1}\xi_{2}}{\left | \bar{\xi} \right | ^{2}}  & 0 \\
	-l\frac{\xi_{2}\xi_{3}}{\left | \bar{\xi} \right | ^{2}}	& l\frac{\xi_{1}\xi_{3}}{\left | \bar{\xi} \right | ^{2}} & \lambda + \left | \bar{\xi} \right | ^{2}
	\end{pmatrix}\begin{pmatrix}
	\widehat{v_{1}} \\
	\widehat{v_{2}} \\
	\widehat{v_{3}}
\end{pmatrix} = \begin{pmatrix}
\widehat{g_{1}} \\
\widehat{g_{2}} \\
\widehat{g_{3}}
\end{pmatrix}.
\end{displaymath}
Setting $ \omega = \sqrt{\lambda + \left | \bar{\xi} \right | ^{2} } $, we have
\begin{displaymath}
	\begin{pmatrix}
	\widehat{v_{1}}	\\
	\widehat{v_{2}}
	\end{pmatrix} = \frac{1}{\mathrm {det}} \begin{pmatrix}
	\omega^{2} + e^{-2\alpha}l\frac{\xi_{1}\xi_{2}}{\left | \bar{\xi} \right | ^{2}} & l\left( 1 - e^{-2\alpha}\frac{\xi_{1}^{2}}{\left | \bar{\xi} \right | ^{2}} \right) \\
	-l\left( 1 - e^{-2\alpha}\frac{\xi_{2}^{2}}{\left | \bar{\xi} \right | ^{2}}\right)  & \omega^{2} - e^{-2\alpha}l\frac{\xi_{1}\xi_{2}}{\left | \bar{\xi} \right | ^{2}} 
\end{pmatrix}   \begin{pmatrix}
\widehat{g_{1}}	\\
\widehat{g_{2}}
\end{pmatrix},
\end{displaymath}
where
\begin{displaymath}
	\mathrm{det} = \omega^{4} + l^{2}\frac{\xi_{3}^{2}}{\left | \bar{\xi} \right | ^{2}}.
\end{displaymath}
 Hence, the solution $ \widehat{v} $ to $ \eqref{2.3} $ is given by
\begin{displaymath}
	\widehat{v} = \frac{\omega^{2}}{\mathrm{det}}I\widehat{g} + \frac{l}{\mathrm{det}}\frac{\xi_{3}}{\left | \bar{\xi} \right |}R\widehat{g},
\end{displaymath}
where $ I $ is the identity matrix and
\begin{displaymath}
	R = R\left ( \xi \right ) = \frac{1}{\left | \bar{\xi} \right |}\begin{pmatrix}
		0 & \xi_{3} & -e^{-2\alpha}\xi_{2} \\
		-\xi_{3} & 0 & e^{-2\alpha}\xi_{1} \\
		\xi_{2} & -\xi_{1} & 0
	\end{pmatrix}.
\end{displaymath}
Next, we give an explicit expression for the solution $ \widehat{v} $ and $ \widehat{\pi} $ of the time dependent linear problem $ \eqref{2.1} $ in Fourier variables, i.e. of the problem
\begin{equation}\label{2.5}
	\begin{split}
			\widehat{v_{t}}\left( \xi\right) + \left | \bar{\xi} \right |^{2}I\widehat{v}\left( \xi\right) + \widehat{le_{3}\times v\left( \xi\right) } + Mi\xi \widehat{\pi}\left( \xi\right) = 0, \quad \xi\in \mathbb{R}^{3},t>0, \\
			i\xi \cdot \widehat{v}\left( \xi\right) = 0,\quad x\in \mathbb{R}^{3},t>0, \\
			\widehat{v}\left( 0,\xi\right) = \widehat{v_{0}}\left( \xi\right) \quad x\in \mathbb{R}^{3}.
	\end{split}
\end{equation}
Observe that
\begin{displaymath}
	\int_{0}^{\infty } e^{-\lambda t}\frac{1}{2} \left [ e^{il\frac{ \xi_{3} }{\left | \bar{\xi}  \right |}t } + e^{-il\frac{ \xi_{3} }{\left | \bar{\xi}  \right |}t } \right ] e^{-\left| \bar{\xi}\right|^{2}t}dt = \frac{1}{2} \left[ \frac{1}{\omega^{2}+il\frac{ \xi_{3} }{\left| \bar{\xi}\right| }} + \frac{1}{\omega^{2}-il\frac{ \xi_{3} }{\left| \bar{\xi}\right| }} \right] = \frac{\omega^{2}}{\mathrm{det}}
\end{displaymath}
and that
\begin{displaymath}
	\frac{l}{\mathrm{det}}\frac{\xi_{3}}{\left| \bar{\xi}\right| } = -\frac{1}{2i}\left[ \frac{1}{\omega^{2}+il\frac{ \xi_{3} }{\left| \bar{\xi}\right| }} - \frac{1}{\omega^{2}-il\frac{ \xi_{3} }{\left| \bar{\xi}\right| }} \right].
\end{displaymath}
Hence, we obtain the following proposition.

\begin{proposition}\label{proposition 2.1}
	There exists a unique solution $ \left( \widehat{v}, \widehat{\pi}\right) $ of equation $ \eqref{2.5} $, where $ \widehat{v} $ is given by
	\begin{displaymath}
		\widehat{v}\left( t,\xi\right) = \cos\left( l\frac{\xi_{3}}{\left| \bar{\xi}\right|}t\right) e^{-\left| \bar{\xi}\right|^{2}t}I\widehat{v_{0}}\left( \xi\right) + \sin\left( l\frac{\xi_{3}}{\left| \bar{\xi}\right|}t \right) e^{-\left| \bar{\xi}\right|^{2}t}R\left( \xi\right) \widehat{v_{0}}\left( \xi\right),\quad t\ge 0,\xi \in \mathbb{R}^{3}.\\
	\end{displaymath}
\end{proposition}
From Proposition 2.1 and Plancherel’s theorem, we obtain the semigroup generated by equation $ \eqref{1.2} $, which is given explicitly by
\begin{displaymath}
	\begin{split}
		T\left( t\right)g := &\mathcal{F}^{-1} \left[  \cos\left( l\frac{\xi_{3}}{\left| \bar{\xi}\right|}t\right) e^{-\left| \bar{\xi}\right|^{2}t}I\widehat{g}\left( \xi\right) + \sin\left( l\frac{\xi_{3}}{\left| \bar{\xi}\right|}t \right) e^{-\left| \bar{\xi}\right|^{2}t}R\left( \xi\right) \widehat{g}\left( \xi\right)\right],\\
		&t \ge 0,g \in L_{\sigma}^{p}\left( \mathbb{R} ^{3}\right).
	\end{split}
\end{displaymath}

\begin{definition}
	A function $ v\left( t,x\right) $ is called a mild solution of $ \eqref{1.2} $, if it satisfies the following integral form:
\begin{equation}\label{2.6}
	v\left ( t \right ) :=-\int_{-\infty}^{t} T\left ( t-\tau   \right ) \mathbb{P} \left [ \left ( v\left ( \tau  \right ) \cdot \nabla \right )v\left ( \tau \right ) \right ]d\tau + \int_{-\infty}^{t} T\left ( t-\tau   \right ) \mathbb{P} f\left( \tau\right) d\tau,
\end{equation}
where $ \mathbb{P} = \left ( \delta_{ij} + R_{i}R_{j} \right ) _{1 \le i,j \le 3} $ denotes the Helmholtz projection onto the divergence-free vector field.
\end{definition}

\begin{definition}
	(i) System $ \eqref{1.1} $ is a $ \left( Q,T\right) $-affine-periodic system, that is, $ f\left( t+T,x \right) = Qf\left( t,Q^{-1}x \right),\pi \left( t+T,x\right) = \pi \left( t,Q^{-1}x\right) $. Specifically, when $ Q $ is an orthogonal matrix, system $ \eqref{1.1} $ is called a rotating periodic system.\\
	(ii) $ f $ is spiral-like almost periodic, that is, for every $ \varepsilon >0 $ there exists a $ T\left( \varepsilon\right) >0 $ such that for any interval $ I $ of length $ T $ there is an $ s $ in $ I $ such that $ \left \| f\left( s+\cdot ,x\right) -Qf\left( \cdot,Q^{-1}x\right) \right \| _{\mathcal{L}} \\<\varepsilon $;\\
	(iii) $ f $ is spiral-like almost automorphic, that is, for each sequence $ \left\lbrace t'_{n}\right\rbrace  $ there exists a subsequence $ \left\lbrace t_{n}\right\rbrace \to \infty $ as $ n \to \infty $ and $ g\in BC\left ( \mathbb{R};{\dot{B}_{p,2} ^{-s} } \left ( \mathbb{R}^{3} \right ) \right )\cap BC\left ( \mathbb{R};{L^{\kappa}} \left ( \mathbb{R}^{3} \right ) \right ) $ such that $ \left \| f\left( \cdot+t_{n} ,x\right) -Qg\left( \cdot, Q^{-1}x\right) \right \| _{\mathcal{L}} \to 0 $ and $ \left \| g\left( \cdot-t_{n} ,x\right) -Qf\left( \cdot,Q^{-1}x\right) \right \| _{\mathcal{L}} \to 0 $ as $ n \to \infty $.\\
	Here, $ \left \| f \right \|_{\mathcal{L} } :=\sup\limits_{\mathbb{R} } \left \| f\left ( \cdot  \right )  \right \|_{\dot{B}_{p,2} ^{-s}} + {\sup\limits_{\mathbb{R} } \left \| f\left ( \cdot  \right )  \right \|_{\kappa}} $, where $ p,s,\kappa >0 $.
\end{definition}

\begin{definition}
	Assume that $ u\left( t\right)  $ is a mild solution of $ \eqref{2.6} $.\\
	(i) $ u $ is $ \left( Q,T\right) $-affine-periodic, that is, $ u\left( t+T,x \right) = Qu\left( t,Q^{-1}x \right) $. Specifically, when $ Q $ is an orthogonal matrix, $ u $ is called the rotating periodic solution.\\
	(ii) $ u $ is a spiral-like almost periodic solution, that is, for every $ \varepsilon >0 $ there exists a $ T\left( \varepsilon\right) >0 $ such that for any interval $ I $ of length $ T $ there is an $ s $ in $ I $ such that $ \left \| u\left( s+\cdot ,x\right) -Qu \right.$ $\left. \left( \cdot,Q^{-1}x\right) \right \| _{X^{r,q}} <\varepsilon $;\\
	(iii) $ u $ is spiral-like almost automorphic, that is, for each sequence $ \left\lbrace t'_{n}\right\rbrace  $ there exists a subsequence $ \left\lbrace t_{n}\right\rbrace \to \infty $ as $ n \to \infty $ and $ v\in BC\left( \mathbb{R}; L_{\sigma }^{r} \left ( \mathbb{R}^{3} \right ) \right) \cap BC\left( \mathbb{R}; \dot{W}^{1,q}\left ( \mathbb{R}^{3} \right ) \right) $ such that $ \left \| u\left( \cdot+t_{n} ,x\right) - Qv \left( \cdot,Q^{-1}x\right) \right \| _{X^{r,q}}$  $ \to 0 $ and $ \left \| v\left( \cdot-t_{n} ,x\right) -Qu\left( \cdot,Q^{-1}x\right) \right \| _{X^{r,q}} \to 0 $ as $ n \to \infty $.\\
	Here, $ X^{r,q} := BC\left ( \mathbb{R};L_{\sigma }^{r} \left ( \mathbb{R}^{3} \right ) \right )\cap BC\left ( \mathbb{R}; \dot{W}^{1,q}\left ( \mathbb{R}^{3} \right ) \right ), \left \| v \right \|_{X^{r,q} } := \sup\limits_{\mathbb{R} }\left \| v\left(\cdot \right)  \right \|_{r} + \sup\limits_{\mathbb{R} }\left \| \nabla v\left( \cdot\right)  \right \|_{q} $, where $ r,q >0 $.
\end{definition}

Similar to Proposition 2.4 in \cite{ref7}. Below we introduce the $ L^{p}-L^{q} $ estimates for the semigroup $ T\left( t\right) $, which play an important role in the subsequent proof of Lemma $ \ref{lemma 3.1} $.
\begin{lemma}\label{lemma 2.1}
	Let $ 1 \le p \le 2 \le q \le \infty $. Then for any $ m \in \mathbb{N}_{0} $ and $ s\ge 0 $ there exists a constant $ C > 0 $
	such that
	\begin{displaymath}
		\begin{split}
			\left \| \nabla^{m} T\left ( t \right ) g  \right \| _{q} \le C t^{-\frac{m}{2}-\frac{3}{2}\left ( \frac{1}{p} -\frac{1}{q}  \right ) }\left \| g \right \|_{p},\quad t > 0,g \in L_{\sigma}^{p}\left( \mathbb{R} ^{3}\right),\\
			\left \| \left( -\Delta  \right) ^{s} T\left ( t \right ) g  \right \| _{2} \le C t^{-s-\frac{3}{2}\left ( \frac{1}{p} -\frac{1}{2}  \right ) }\left \| g \right \|_{p}, \quad t > 0,g \in L_{\sigma}^{p}\left( \mathbb{R} ^{3}\right),
		\end{split}
	\end{displaymath}
	and the same assertions hold for $ T^{\ast}\left( t\right) $, which denotes the adjoint operator of $ T\left( t\right) $.
\end{lemma}
\begin{proof}
	Note the semigroup $ T $ can be rewritten as
	\begin{displaymath}
		T\left( t \right) g = e^{\frac{t}{2}\bar{\Delta }}\left [ \cos\left ( lR_{3}t \right )I + \sin \left ( lR_{3}t \right )R \right ]e^{\frac{t}{2}\bar{\Delta }}g,
	\end{displaymath}
	where $ R_{3} $ is defined by $ \widehat{R_{3}g}\left( \xi \right) := \frac{\xi_{3}}{\left| \bar{\xi}\right| }\widehat{g}\left( \xi\right) $. Similar to the classical $ L^{p}-L^{q} $ smooth properties of $ e^{t\Delta} $, we have
	\begin{displaymath}
		\begin{split}
			\left \| \nabla^{m} T\left ( t \right ) g  \right \| _{q} 
			&\le \left \| \left [ \cos\left ( lR_{3}t \right )I + \sin \left ( lR_{3}t \right )R \right ]\nabla^{m}e^{\frac{t}{2}\bar{\Delta }}e^{\frac{t}{2}\bar{\Delta }}g  \right \|_{q} \\
			&\le C\left \| \nabla^{m}e^{\frac{t}{2}\bar{\Delta }}e^{\frac{t}{2}\bar{\Delta }}g  \right \|_{q} \le C t^{-\frac{m}{2}-\frac{3}{2}\left ( \frac{1}{p} -\frac{1}{q}  \right ) }\left \| g \right \|_{p},
		\end{split}
	\end{displaymath}
	\begin{displaymath}
		\begin{split}
			\left \| \left( -\Delta  \right) ^{s} T\left ( t \right ) g  \right \| _{2} 
			&\le \left \| \left [ \cos\left ( lR_{3}t \right )I + \sin \left ( lR_{3}t \right )R \right ]\left( -\Delta  \right) ^{s}e^{\frac{t}{2}\bar{\Delta }}e^{\frac{t}{2}\bar{\Delta }}g  \right \| _{2} \\
			&\le \left \|\left( -\Delta  \right) ^{s}e^{\frac{t}{2}\bar{\Delta }}e^{\frac{t}{2}\bar{\Delta }}g  \right \| _{2} \le C t^{-s-\frac{3}{2}\left ( \frac{1}{p} -\frac{1}{2}  \right ) }\left \| g \right \|_{p}.
		\end{split}
	\end{displaymath}
\end{proof}

Similar to the approach of Lemma 2.2 in \cite{ref14}, We have the following lemma.
\begin{lemma}\label{lemma 2.2}
	Let $ -\infty <s_{0} \le s_{1} <\infty $.  Then there exists a positive constant $ C=C\left( s_{0},s_{1}\right) $ such that
	\begin{equation}\label{2.7}
		\left \| e^{t\bar{\Delta}}f \right \| _{\dot{B}_{p,q} ^{s_{1}}} \le Ct^{-\frac{1}{2}\left( s_{1}-s_{0}\right) }\left \| f \right \| _{\dot{B}_{p,q} ^{s_{0}}}
	\end{equation}
	for all $ t>0,1 \le p,q \le \infty $ and $ f\in {\dot{B}_{p,q} ^{s_{0}} } \left ( \mathbb{R}^{3} \right ) $.
\end{lemma}
\begin{proof}
	For $ 1 \le p,q\le \infty $,
	\begin{displaymath}
		\left \| f \right \| _{\dot{B}_{p,q} ^{s}} := \left \| \left \{ 2^{sj}\left \| \varphi _{j}\ast f \right \|_{p}  \right \} _{j \in \mathbb{Z} } \right \| _{q} = \left ( \sum_{j=-\infty}^{\infty} \left ( 2^{sj}\left \| \varphi _{j}\ast f \right \|_{L^{p}} \right ) ^{q} \right )^{\frac{1}{q}}.
	\end{displaymath}
	Let 
	\begin{displaymath}
		G\left( t,x\right) = \frac{e^{2\alpha}}{\left( 4\pi t \right)^{\frac{3}{2}}}e^{-\frac{e^{2\alpha}x_{1}^{2}+e^{2\alpha}x_{2}^{2}+x_{3}^{2}}{4t}}.
	\end{displaymath}  From the Fourier transform and the inverse transform we have
	\begin{displaymath}
		e^{t\bar{\Delta}}f = G\left( t,x\right) \ast f.
	\end{displaymath}
	To prove $ \eqref{2.7} $, it suffices to show that 
	\begin{equation}\label{2.8}
		2^{s_{1}j}\left \| \varphi_{j} \ast G \ast f \right \| _{L^{p}}\le Ct^{-\frac{1}{2}\left( s_{1}-s_{0}\right)}2^{s_{0}j}\left \| \varphi_{j} \ast f \right \|_{L^{p}}
	\end{equation}
	with a constant $ C=C\left( s_{0},s_{1}\right) $ independent of $ t>0,j\in \mathbb{Z} ,1\le p \le \infty $. Since $ \mathrm{supp}\hat{\varphi } \subset \left \{ \xi \in \mathbb{R}^{3} \mid 2^{-1}\le \left | \xi  \right |\le 2 \right \} $, $ \mathrm{supp}\hat{\varphi_{j}} \subset \left \{ \xi \in \mathbb{R}^{3} \mid 2^{j-1}\le \left | \xi  \right |\le 2^{j+1} \right \} $, we have
	\begin{displaymath}
		\varphi _{j}\ast G \ast f = \tilde{\varphi_{j}}\ast G \ast \varphi_{j} \ast f 
	\end{displaymath}
	for all $ j\in \mathbb{Z} $, where $ \tilde{\varphi_{j}} \equiv \varphi _{j-1}+\varphi _{j}+\varphi _{j+1} $, and hence by the Young inequality, there holds
	\begin{displaymath}
		\begin{split}
			2^{s_{1}j}\left \| \varphi_{j} \ast G \ast f \right \| _{L^{p}} &= 2^{\beta j}2^{s_{0}j}\left \| \left( -\Delta\right) ^{-\alpha/2} \tilde{\varphi_{j}} \ast \left( -\Delta\right) ^{\alpha/2} G \ast \varphi_{j} \ast f \right \| _{L^{p}} \\
			&\le 2^{\beta j}\left \| \left( -\Delta\right) ^{-\alpha/2} \tilde{\varphi_{j}} \right \| _{L^{1}}\left \| \left( -\Delta\right) ^{\alpha/2} G \right \| _{L^{1}}\left( 2^{s_{0}j}\left \| \varphi_{j} \ast f \right \| _{L^{p}}\right),
		\end{split}
	\end{displaymath}
	where $ \beta \equiv s_{1}-s_{0} $. From Lemma $ \ref{lemma 2.1} $ it is easy to see that
	\begin{displaymath}
		\left \| \left ( -\Delta  \right ) ^{-\alpha /2}\tilde {\varphi} _{j}\right \| _{L^{1}} = C2^{-\alpha j},\left \| \left ( -\Delta  \right ) ^{\alpha /2}G\right \| _{L^{1}} \le Ct^{-\alpha/2}
	\end{displaymath}
	for all $ t>0 $ and all $ j \in \mathbb{Z} $ with $ C = C\left( s_{0},s_{1}\right) $. From this and the above estimate, we obtain $ \eqref{2.8} $.
\end{proof}

In the next section, we consider the existence and uniqueness of the rotating periodic, spiral-like almost periodic and spiral-like almost automorphic solutions of equation $ \eqref{1.1} $ when $ \alpha =0 $.

\section{The rotating solution, spiral-like almost periodic solution and spiral-like almost automorphic solution}

Below we give a main estimate of the equation \eqref{1.2}. Specifically, when $ \alpha = 0 $, the following lemma is the Lemma 3.1 in \cite{ref16}.

\begin{lemma}\label{lemma 3.1}
	Let $ 2\le r< 3 $ and $ 2<q\le 12/5 $. Assume that the exponents $ s,p $ and $ \kappa $ satisfy
	\begin{displaymath}
		s>0,1<p\le 2,\frac{s}{3}+\frac{1}{p}> \frac{1}{r}+\frac{2}{3} , \frac{1}{2}\le\frac{1}{\kappa}< \frac{1}{q}+\frac{1}{3}.
	\end{displaymath}	
 Then, if $ n>N>0 $, for every $ f\in BC\left ( \mathbb{R};{\dot{B}_{p,2} ^{-s} } \left ( \mathbb{R}^{3} \right ) \right )\cap BC\left ( \mathbb{R};{L^{\kappa}} \left ( \mathbb{R}^{3} \right ) \right )  $ satisfying $ \left \| f \right \|_{\mathcal{L} } \le \delta $, we have
	\begin{displaymath}
		\left \| v\left ( t \right )  \right \| _{X^{r,q} }\le C_{\ast } \left \| v\left ( t \right )  \right \|^{2}_{X^{r,q}} + C_{\ast \ast }\left ( \left \| f \right \|_{\mathcal{L} }  + \left \| f \right \| _{ \mathcal{K} _{N}}  \right )  ,
	\end{displaymath}	
	where $ \left \| f  \right \| _{ \mathcal{K} _{N}}= \sup\limits_{\left | t \right | \le  N} \left \| f\left ( t \right )  \right \|_{\dot{B}_{p,2} ^{-s} } + \sup\limits_{\left | t \right | \le  N} \left \| f \left ( t \right )  \right \|_{\kappa } $ and $ C_{\ast },C_{\ast \ast } $ are constants.
\end{lemma}
\begin{proof}
	Let $ n>N>0,t \in \left[ -N,N \right] $ and $ v \in BC\left( \mathbb{R};L_{\sigma }^{r} \left ( \mathbb{R}^{3} \right ) \right) \cap BC\left( \mathbb{R}; \dot{W}^{1,q} \left ( \mathbb{R}^{3} \right ) \right)  $. We have
	\begin{equation}\label{3.1}
		\begin{split}
			N\left( t\right) &= \int_{-\infty }^{t} T\left ( t-\tau \right ) \mathbb{P} \left [ \left ( u\left ( \tau  \right ) \cdot \nabla \right )v\left ( \tau \right ) \right ] d\tau \\
			& = \int_{-\infty }^{-n} T\left ( t-\tau \right ) \mathbb{P} \left [ \left ( u\left ( \tau  \right ) \cdot \nabla \right )v\left ( \tau \right ) \right ] d\tau \\
			&\quad + \int_{-n}^{t-1} T\left ( t-\tau \right ) \mathbb{P} \left [ \left ( u\left ( \tau  \right ) \cdot \nabla \right )v\left ( \tau \right ) \right ] d\tau \\
			&\quad + \int_{t-1}^{t} T\left ( t-\tau \right ) \mathbb{P} \left [ \left ( u\left ( \tau  \right ) \cdot \nabla \right )v\left ( \tau \right ) \right ] d\tau \\
			& := N_{1}\left( t\right) + N_{2}\left( t\right) + N_{3}\left( t\right).
		\end{split}
	\end{equation}
	Since $ 2\le r<3 $, there hold that $ 3/2 <{r}' \le 2 $ and $ 3<{\left( r/2 \right) }' \le \infty $, where $ \frac{1}{r} + \frac{1}{{r}'} = 1,\frac{1}{r/2}+ \frac{1}{{\left( r/2\right) }'} = 1 $. By integration by parts, the Hölder inequality and Lemma $ \ref{lemma 2.1} $, for all $ \varphi \in C_{0,\sigma }^{\infty } \left ( \mathbb{R}^{3} \right ) $, we have
	\begin{equation*}
		\begin{split}
			\left | \left \langle N_{1}\left ( t \right ),\varphi \right \rangle \right | 
			&= \left | \int_{-\infty }^{-n} \left \langle T\left ( t-\tau \right )\mathbb{P} \left [ \left ( u\left ( \tau  \right ) \cdot \nabla \right )v\left ( \tau \right ) \right ], \varphi \right \rangle \mathrm{d}\tau \right |  \\
			&= \left | \int_{-\infty }^{-n} \left \langle \mathbb{P} \left [ \left ( u\left ( \tau  \right ) \cdot \nabla \right )v\left ( \tau \right ) \right ],T^{\ast }\left ( t-\tau \right ) \varphi \right \rangle \mathrm{d}\tau \right |  \\
			&\le \int_{-\infty }^{-n} \left | \left \langle v\left ( \tau \right ) , \left ( u\left ( \tau  \right ) \cdot \nabla \right ) T^{\ast }\left ( t-\tau \right ) \varphi \right \rangle \right | \mathrm{d}\tau   \\
			&\le \int_{-\infty }^{-n} \left \| v\left ( \tau \right ) \right \|_{r} \left \| \left ( u\left ( \tau  \right ) \cdot \nabla \right ) T^{\ast }\left ( t-\tau \right ) \varphi \right \|_{r'} \mathrm{d}\tau \\
			&\le \int_{-\infty }^{-n} \left \| v\left ( \tau \right ) \right \|_{r} \left \| u\left ( \tau \right ) \right \|_{r}\left \| \nabla T^{\ast }\left ( t-\tau \right ) \varphi \right \|_{\varsigma } \mathrm{d}\tau \\
			&\le C\int_{-\infty}^{-n}\left ( t-\tau \right )^{-\frac{1}{2}-\frac{3}{2}\left ( \frac{1}{{r}'} -\frac{1}{\varsigma} \right )}\left\| \varphi\right\|_{{r}'}\left\| u\left( \tau \right) \right\|_{r}\left\| v\left( \tau \right) \right\|_{r} \mathrm{d}\tau \\
			&\le C \sup_{t\in \mathbb{R}} \left\| u\left( t \right)\right\|_{r} \sup_{t \in \mathbb{R}} \left\| v\left( t \right) \right\|_{r} \left\| \varphi\right\|_{{r}'} \int_{-\infty}^{-n}\left ( t-\tau \right )^{-\frac{1}{2}-\frac{3}{2r}}d\tau \\
			&\le \frac{2r}{3-r}\left ( t+n \right )^{\frac{1}{2}-\frac{3}{2r}}C \sup_{t\in \mathbb{R}} \left\| u\left( t \right)\right\|_{r} \sup_{t\in \mathbb{R}} \left\| v\left( t \right) \right\|_{r} \left\| \varphi\right\|_{{r}'}
		\end{split}
	\end{equation*}
	where $ \frac{1}{r}+\frac{1}{\varsigma }=\frac{1}{r'} $. By the duality of $ L_{\sigma}^{r}\left( \mathbb{R}^{3}\right) $ and $ L_{\sigma}^{{r}'}\left( \mathbb{R}^{3}\right) $, we have
	\begin{equation}\label{3.2}
		\begin{split}
			\left \| N_{1}\left ( t \right )  \right \|_{r} &\le \frac{2r}{3-r}\left ( t+n \right )^{\frac{1}{2}-\frac{3}{2r}}C \sup_{t\in \mathbb{R}} \left\| u\left( t \right)\right\|_{r} \sup_{t\in \mathbb{R}} \left\| v\left( t \right) \right\|_{r} \\
			& \le \frac{2r}{3-r}\left ( n-N \right )^{\frac{1}{2}-\frac{3}{2r}}C \sup_{t\in \mathbb{R}} \left\| u\left( t \right)\right\|_{r} \sup_{t\in \mathbb{R}} \left\| v\left( t \right) \right\|_{r}.
		\end{split}
	\end{equation}
	Similarly, we have
	\begin{equation}\label{3.3}
		\begin{split}
			\left \| N_{2}\left ( t \right )  \right \|_{r} &\le C \sup_{\left| t\right| \le n} \left\| u\left( t \right)\right\|_{r} \sup_{\left| t\right| \le n} \left\| v\left( t \right) \right\|_{r} \int_{-n}^{t-1}\left ( t-\tau \right )^{-\frac{1}{2}-\frac{3}{2r}}d\tau \\
			&= \frac{2r}{3-r}\left[ 1-\left ( t+n \right )^{\frac{1}{2}-\frac{3}{2r}}\right] C \sup_{\left| t\right| \le n} \left\| u\left( t \right)\right\|_{r} \sup_{\left| t\right| \le n} \left\| v\left( t \right) \right\|_{r} \\
			& \le C \sup_{\left| t\right| \le n} \left\| u\left( t \right)\right\|_{r} \sup_{\left| t\right| \le n} \left\| v\left( t \right) \right\|_{r}.
		\end{split}
	\end{equation}
 By Hölder inequality, we have
	\begin{equation}\label{3.4}
		\begin{split}
			\left \| N_{3}\left ( t \right )  \right \|_{r} &\le \int_{t-1}^{t} \left\|  T\left ( t-\tau \right ) \mathbb{P} \left [ \left ( u\left ( \tau  \right ) \cdot \nabla \right )v\left ( \tau \right ) \right ]\right\|_{r} d\tau \\
			&\le C \int_{t-1}^{t} \left ( t-\tau \right )^{-\frac{3}{2}\left( \frac{1}{k}-\frac{1}{r}\right)} \left\|  \left ( u\left ( \tau  \right ) \cdot \nabla \right )v\left ( \tau \right )\right\|_{k} d\tau \\
			&\le C \int_{t-1}^{t} \left ( t-\tau \right )^{-\frac{3}{2q}} \left\| u\left( \tau \right) \right\|_{r} \left\| \nabla v\left ( \tau \right )\right\|_{q} d\tau \\
			&\le C \sup_{\left| t\right| \le n}\left\| u\left( t\right) \right\|_{r} \sup_{\left| t\right| \le n}\left\| \nabla v\left ( t \right )\right\|_{q}\int_{t-1}^{t} \left ( t-\tau \right )^{-\frac{3}{2q}}d\tau \\
			&\le C \sup_{\left| t\right| \le n}\left\| u\left( t\right) \right\|_{r} \sup_{\left| t\right| \le n}\left\| \nabla v\left ( t \right )\right\|_{q},
		\end{split}
	\end{equation}
	where $ \frac{1}{k}:=\frac{1}{r} + \frac{1}{q} $ and $ 1<k<\frac{4}{3} $.\\
	By Lemma $ \ref{lemma 2.1} $ and Hölder inequality, we have
	\begin{equation}\label{3.5}
		\begin{split}
			\left \| \nabla N_{1}\left ( t \right )  \right \|_{q} 
			&\le \int_{-\infty}^{-n} \left\| \nabla T\left ( t-\tau \right ) \mathbb{P} \left [ \left ( u\left ( \tau  \right ) \cdot \nabla \right )v\left ( \tau \right ) \right ]\right\|_{q} d\tau \\
			&\le C \int_{-\infty}^{-n} \left ( t-\tau \right )^{-\frac{1}{2}-\frac{3}{2}\left( \frac{1}{k}-\frac{1}{q}\right)} \left\|  \left ( u\left ( \tau  \right ) \cdot \nabla \right )v\left ( \tau \right )\right\|_{k} d\tau \\
			&\le C \int_{-\infty}^{-n} \left ( t-\tau \right )^{-\frac{1}{2}-\frac{3}{2r}} \left\| u\left( \tau \right) \right\|_{r} \left\| \nabla v\left ( \tau \right )\right\|_{q} d\tau \\ 
			&\le \frac{2r}{3-r}\left ( t+n \right )^{\frac{1}{2}-\frac{3}{2r}}C \sup_{t\in \mathbb{R}} \left\| u\left( t \right)\right\|_{r} \sup_{t\in \mathbb{R}} \left\| \nabla v\left( t \right) \right\|_{q} \\
			& \le \frac{2r}{3-r}\left ( n-N \right )^{\frac{1}{2}-\frac{3}{2r}}C \sup_{t\in \mathbb{R}} \left\| u\left( t \right)\right\|_{r} \sup_{t\in \mathbb{R}} \left\| \nabla v\left( t \right) \right\|_{q},
		\end{split}
	\end{equation}
	\begin{equation}\label{3.6}
		\begin{split}
			\left \| \nabla N_{2}\left ( t \right )  \right \|_{q} 
			&\le C \int_{-n}^{t-1} \left ( t-\tau \right )^{-\frac{1}{2}-\frac{3}{2}\left( \frac{1}{k}-\frac{1}{q}\right)} \left\|  \left ( u\left ( \tau  \right ) \cdot \nabla \right )v\left ( \tau \right )\right\|_{k} d\tau \\
			&\le C \int_{-n}^{t-1} \left ( t-\tau \right )^{-\frac{1}{2}-\frac{3}{2r}} \left\| u\left( \tau \right) \right\|_{r} \left\| \nabla v\left ( \tau \right )\right\|_{q} d\tau \\
			&\le C \sup_{\left| t\right| \le n} \left\| u\left( t \right)\right\|_{r} \sup_{\left| t\right| \le n} \left\| \nabla v\left( t \right) \right\|_{q}.
		\end{split}
	\end{equation}
	It follows from the Sobolev embedding theorem $ \dot{W}^{1,q}\left ( \mathbb{R}^{3} \right )\hookrightarrow L^{q^{\ast}}\left ( \mathbb{R}^{3} \right ) $ that
	\begin{equation}\label{3.7}
		\begin{split}
			\left \| \nabla N_{3}\left ( t \right )  \right \|_{q} 
			&\le \int_{t-1}^{t} \left\| \nabla T\left ( t-\tau \right ) \mathbb{P} \left [ \left ( u\left ( \tau  \right ) \cdot \nabla \right )v\left ( \tau \right ) \right ]\right\|_{q} d\tau \\
			&\le C \int_{t-1}^{t} \left ( t-\tau \right )^{-\frac{1}{2}-\frac{3}{2}\left( \frac{1}{\iota }-\frac{1}{q}\right)} \left\|  \left ( u\left ( \tau  \right ) \cdot \nabla \right )v\left ( \tau \right )\right\|_{\iota} d\tau \\
			&\le C \int_{t-1}^{t} \left ( t-\tau \right )^{-\frac{1}{2}-\frac{3}{2q^{\ast}}} \left\| u\left( \tau \right) \right\|_{q^{\ast}} \left\| \nabla v\left ( \tau \right )\right\|_{q} d\tau \\ 
			&\le C \sup_{\left| t\right| \le n} \left\| \nabla u\left( t \right)\right\|_{q} \sup_{\left| t\right| \le n} \left\| \nabla v\left( t \right) \right\|_{q},
		\end{split}
	\end{equation}
	where $ \frac{1}{q^{\ast}} = \frac{1}{q}-\frac{1}{3},\frac{1}{\iota} = \frac{1}{q^{\ast}}+\frac{1}{q} $.
	Set
	\begin{equation}\label{3.8}
		\begin{split}
			F\left( t\right) &= \int_{-\infty}^{-n} T\left ( t-\tau \right ) \mathbb{P} f\left ( \tau \right ) d\tau + \int_{-n}^{t-1} T\left ( t-\tau \right ) \mathbb{P} f\left ( \tau \right ) d\tau + \int_{t-1}^{t} T\left ( t-\tau \right ) \mathbb{P} f\left ( \tau \right ) d\tau \\
			&:= F_{1}\left( t\right) + F_{2}\left( t\right) + F_{3}\left( t\right) 
		\end{split}
	\end{equation}
	for $ t \in \mathbb{R} $. Since $ 1<p\le 2\le r<3 $, by Lemma $ \ref{lemma 2.1} $ and Lemma $ \ref{lemma 2.2} $, for $ n>N $, we have 
	\begin{equation}\label{3.9}
		\begin{split}
			\left \| F_{1}\left ( t \right )  \right \|_{r} 
			&\le C\int_{-\infty }^{-n} \left ( t-\tau \right )^{-\frac{3}{2}\left ( \frac{1}{2}-\frac{1}{r} \right ) }\left \| T\left ( \frac{t-\tau}{2} \right ) \mathbb{P} f\left ( \tau \right ) \right \|_{2} d\tau \\
			&\le C\int_{-\infty }^{-n} \left ( t-\tau \right )^{-\frac{3}{2}\left ( \frac{1}{2}-\frac{1}{r} \right ) }\left \| e^{\frac{t-\tau}{2}\bar{\Delta}} \mathbb{P} f\left ( \tau \right ) \right \|_{2} d\tau \\
			&\le C\int_{-\infty }^{-n} \left ( t-\tau \right )^{-\frac{3}{2}\left ( \frac{1}{2}-\frac{1}{r} \right )-\frac{s}{2} }\left \| e^{\frac{t-\tau}{2}\bar{\Delta}} \mathbb{P} f\left ( \tau \right ) \right \|_{\dot{B}_{2,2} ^{-s}} d\tau \\
			&\le C\int_{-\infty }^{-n} \left ( t-\tau \right )^{-\frac{3}{2}\left ( \frac{1}{2}-\frac{1}{r} \right )-\frac{s}{2}-\frac{3}{2}\left ( \frac{1}{p}-\frac{1}{2} \right ) }\left \| f\left ( \tau \right ) \right \|_{\dot{B}_{p,2} ^{-s}} d\tau \\
			&\le C\int_{-\infty }^{-n} \left ( t-\tau \right )^{-\frac{3}{2}\left ( \frac{1}{p}-\frac{1}{r} \right )-\frac{s}{2}} d\tau \sup_{t\in \mathbb{R}}\left \| f\left ( t\right ) \right \|_{\dot{B}_{p,2} ^{-s}} \\
			&\le \frac{C}{\frac{3}{2}\left ( \frac{1}{p}-\frac{1}{r} \right )+\frac{s}{2}-1}\left ( t+n \right )^{1-\frac{3}{2}\left ( \frac{1}{p}-\frac{1}{r} \right )-\frac{s}{2}} \sup_{t\in \mathbb{R}}\left \| f\left ( t\right ) \right \|_{\dot{B}_{p,2} ^{-s}} \\
			&\le C \left ( n-N \right )^{1-\frac{3}{2}\left ( \frac{1}{p}-\frac{1}{r} \right )-\frac{s}{2}} \sup_{t\in \mathbb{R}}\left \| f\left ( t\right ) \right \|_{\dot{B}_{p,2} ^{-s}},
		\end{split}
	\end{equation}
	where $ 1-\frac{3}{2}\left ( \frac{1}{p}-\frac{1}{r} \right )-\frac{s}{2}<0 $ follows from the hypothesis on $ s $.\\
	Similarly, we have
	\begin{equation}\label{3.10}
		\begin{split}
			\left \| F_{2}\left ( t \right )  \right \|_{r} 
			&\le C\int_{-n}^{t-1} \left ( t-\tau \right )^{-\frac{3}{2}\left ( \frac{1}{2}-\frac{1}{r} \right ) }\left \| T\left ( \frac{t-\tau}{2} \right ) \mathbb{P} f\left ( \tau \right ) \right \|_{2} d\tau \\
			&\le C\int_{-n}^{t-1} \left ( t-\tau \right )^{-\frac{3}{2}\left ( \frac{1}{2}-\frac{1}{r} \right ) }\left \| e^{\frac{t-\tau}{2}\bar{\Delta}} \mathbb{P} f\left ( \tau \right ) \right \|_{2} d\tau \\
			&\le C\int_{-n}^{t-1} \left ( t-\tau \right )^{-\frac{3}{2}\left ( \frac{1}{2}-\frac{1}{r} \right )-\frac{s}{2}-\frac{3}{2}\left ( \frac{1}{p}-\frac{1}{2} \right ) }\left \| f\left ( \tau \right ) \right \|_{\dot{B}_{p,2} ^{-s}} d\tau \\
			&\le C\int_{-n}^{t-1} \left ( t-\tau \right )^{-\frac{3}{2}\left ( \frac{1}{p}-\frac{1}{r} \right )-\frac{s}{2}} d\tau \sup_{\left| t\right| \le n}\left \| f\left ( t\right ) \right \|_{\dot{B}_{p,2} ^{-s}} \\
			&\le C \sup_{\left| t\right| \le n}\left \| f\left ( t\right ) \right \|_{\dot{B}_{p,2} ^{-s}}.
		\end{split}
	\end{equation}
	Since $ \frac{6}{5} < \kappa \le 2\le r <3 $, by Lemma $ \ref{lemma 2.1} $, we have
	\begin{equation}\label{3.11}
		\begin{split}
			\left \| F_{3}\left ( t \right )  \right \|_{r} 
			&\le C\int_{t-1}^{t} \left ( t-\tau \right )^{-\frac{3}{2}\left ( \frac{1}{\kappa}-\frac{1}{r} \right ) }\left \| f\left ( \tau \right ) \right \|_{\kappa} d\tau \\
			&\le C\int_{t-1}^{t} \left ( t-\tau \right )^{-\frac{3}{2}\left ( \frac{1}{\kappa}-\frac{1}{r} \right ) }d\tau \sup_{\left| t\right| \le n}\left \| f\left ( t \right ) \right \|_{\kappa}  \\
			&\le C \sup_{\left| t\right| \le n}\left \| f\left ( t \right ) \right \|_{\kappa}.
		\end{split}
	\end{equation}
	Here, we remark that since $ r<3 $ and  $ q>2 $, it follows from the assumption on $ \kappa $ that $ \frac{1}{\kappa} < \frac{1}{q}+\frac{1}{3} < \frac{1}{2}+\frac{1}{r} < \frac{2}{3}+\frac{1}{r} $. Hence, $ \int_{t-1}^{t} \left ( t-\tau \right )^{-\frac{3}{2}\left ( \frac{1}{\kappa}-\frac{1}{r} \right ) } d\tau $ converges.\\
	Since $ 1<p \le 2<q \le 12/5 $, by Lemma $ \ref{lemma 2.1} $ and Lemma $ \ref{lemma 2.2} $, we have
	\begin{equation}\label{3.12}
		\begin{split}
			\left \| \nabla F_{1}\left ( t \right )  \right \|_{q} 
			&\le C\int_{-\infty }^{-n} \left ( t-\tau \right )^{-\frac{3}{2}\left ( \frac{1}{2}-\frac{1}{q} \right )-\frac{1}{2} }\left \| T\left ( \frac{t-\tau}{2} \right ) \mathbb{P} f\left ( \tau \right ) \right \|_{2} d\tau \\
			&\le C\int_{-\infty }^{-n} \left ( t-\tau \right )^{-\frac{3}{2}\left ( \frac{1}{2}-\frac{1}{q} \right )-\frac{1}{2} }\left \| e^{\frac{t-\tau}{2}\bar{\Delta}} \mathbb{P} f\left ( \tau \right ) \right \|_{2} d\tau \\
			&\le C\int_{-\infty }^{-n} \left ( t-\tau \right )^{-\frac{3}{2}\left ( \frac{1}{2}-\frac{1}{q} \right )-\frac{1}{2}-\frac{s}{2} }\left \| e^{\frac{t-\tau}{2}\bar{\Delta}} \mathbb{P} f\left ( \tau \right ) \right \|_{\dot{B}_{2,2} ^{-s}} d\tau \\
			&\le C\int_{-\infty }^{-n} \left ( t-\tau \right )^{-\frac{3}{2}\left ( \frac{1}{2}-\frac{1}{q} \right )-\frac{1}{2}-\frac{s}{2}-\frac{3}{2}\left ( \frac{1}{p}-\frac{1}{2} \right ) }\left \| f\left ( \tau \right ) \right \|_{\dot{B}_{p,2} ^{-s}} d\tau \\
			&\le C\int_{-\infty }^{-n} \left ( t-\tau \right )^{-\frac{3}{2}\left ( \frac{1}{p}-\frac{1}{q} \right )-\frac{1}{2}-\frac{s}{2}} d\tau \sup_{t\in \mathbb{R}}\left \| f\left ( t\right ) \right \|_{\dot{B}_{p,2} ^{-s}} \\
			&\le \frac{C}{\frac{3}{2}\left ( \frac{1}{p}-\frac{1}{q} \right )+\frac{s}{2}-\frac{1}{2}}\left ( t+n \right )^{\frac{1}{2}-\frac{3}{2}\left ( \frac{1}{p}-\frac{1}{q} \right )-\frac{s}{2}} \sup_{t\in \mathbb{R}}\left \| f\left ( t\right ) \right \|_{\dot{B}_{p,2} ^{-s}} \\
			&\le C \left ( n-N \right )^{\frac{1}{2}-\frac{3}{2}\left ( \frac{1}{p}-\frac{1}{q} \right )-\frac{s}{2}} \sup_{t\in \mathbb{R}}\left \| f\left ( t\right ) \right \|_{\dot{B}_{p,2} ^{-s}},
		\end{split}
	\end{equation}
	where $ \frac{1}{2}-\frac{3}{2}\left ( \frac{1}{p}-\frac{1}{q} \right )-\frac{s}{2} = -\frac{3}{2}\left( \frac{s}{3}+\frac{1}{p}-\frac{1}{q}-\frac{1}{3} \right) < -\frac{3}{2}\left( \frac{1}{r}+\frac{2}{3}-\frac{1}{q}-\frac{1}{3} \right) <0 $.  Similarly,
	\begin{equation}\label{3.13}
		\begin{split}
			\left \| \nabla F_{2}\left ( t \right )  \right \|_{q} 
			&\le C\int_{-n}^{t-1} \left ( t-\tau \right )^{-\frac{3}{2}\left ( \frac{1}{2}-\frac{1}{q} \right )-\frac{1}{2} }\left \| T\left ( \frac{t-\tau}{2} \right ) \mathbb{P} f\left ( \tau \right ) \right \|_{2} d\tau \\
			&\le C\int_{-n}^{t-1} \left ( t-\tau \right )^{-\frac{3}{2}\left ( \frac{1}{2}-\frac{1}{q} \right )-\frac{1}{2} }\left \| e^{\frac{t-\tau}{2}\bar{\Delta}} \mathbb{P} f\left ( \tau \right ) \right \|_{2} d\tau \\
			&\le C\int_{-n}^{t-1} \left ( t-\tau \right )^{-\frac{3}{2}\left ( \frac{1}{2}-\frac{1}{q} \right )-\frac{1}{2}-\frac{s}{2}-\frac{3}{2}\left ( \frac{1}{p}-\frac{1}{2} \right ) }\left \| f\left ( \tau \right ) \right \|_{\dot{B}_{p,2} ^{-s}} d\tau \\
			&\le C\int_{-n}^{t-1} \left ( t-\tau \right )^{-\frac{3}{2}\left ( \frac{1}{p}-\frac{1}{q} \right )-\frac{1}{2}-\frac{s}{2}} d\tau \sup_{\left| t\right| \le n}\left \| f\left ( t\right ) \right \|_{\dot{B}_{p,2} ^{-s}} \\
			&\le C \sup_{\left| t\right| \le n}\left \| f\left ( t\right ) \right \|_{\dot{B}_{p,2} ^{-s}}.
		\end{split}
	\end{equation}
	Since $ \frac{6}{5} < \kappa \le 2 <q \le 12/5 $, by Lemma $ \ref{lemma 2.1} $, we have
	\begin{equation}\label{3.14}
		\begin{split}
			\left \| \nabla F_{3}\left ( t \right )  \right \|_{q} 
			&\le C\int_{t-1}^{t} \left ( t-\tau \right )^{-\frac{3}{2}\left ( \frac{1}{\kappa}-\frac{1}{q} \right )-\frac{1}{2} }\left \| f\left ( \tau \right ) \right \|_{\kappa} d\tau \\
			&\le C\int_{t-1}^{t} \left ( t-\tau \right )^{-\frac{3}{2}\left ( \frac{1}{\kappa}-\frac{1}{q} \right )-\frac{1}{2} }d\tau \sup_{\left| t\right| \le n}\left \| f\left ( t \right ) \right \|_{\kappa}  \\
			&\le C \sup_{\left| t\right| \le n}\left \| f\left ( t \right ) \right \|_{\kappa}.
		\end{split}
	\end{equation}
	Since $ \frac{1}{\kappa} < \frac{1}{q}+\frac{1}{3} $, $ \int_{t-1}^{t} \left ( t-\tau \right )^{-\frac{3}{2}\left ( \frac{1}{\kappa}-\frac{1}{q} \right )-\frac{1}{2} }d\tau $ converges. 
	
	By $ \eqref{3.1}-\eqref{3.14} $, we have
	\begin{equation*}
		\begin{split}
			\left \| v\left ( t \right )  \right \| _{X^{r,q}} 
			&\le C\left( n\right) \left \{ \left ( \sup_{t\in \mathbb{R}}\left \| v\left ( t \right )  \right \| _{r} \right )^{2} + \left ( \sup_{t\in \mathbb{R}}\left \| \nabla v\left ( t \right )  \right \| _{q} \right )^{2} \right \} \\
			&\quad+ C \left \{ \left ( \sup_{\left| t\right| \le n}\left \| v\left ( t \right )  \right \| _{r} \right )^{2} + \left ( \sup_{\left| t\right| \le n}\left \| \nabla v\left ( t \right )  \right \| _{q} \right )^{2} \right \} \\
			&\quad+ \tilde{C} \left( n\right) \left \| f \right \| _{\mathcal{L}} + \tilde{C}\left \| f \right \| _{\mathcal{K}_{n}},
		\end{split}
	\end{equation*}
	where $ C,\tilde{C} $ are constants, and 
	\begin{displaymath}
		C\left( n\right) = \frac{2rC}{3-r}\left ( n-N \right )^{\frac{1}{2}-\frac{3}{2r}},
	\end{displaymath}
	\begin{displaymath}
		\tilde{C}\left( n\right) = \frac{C}{\frac{3}{2}\left ( \frac{1}{p}-\frac{1}{r} \right )+\frac{s}{2}-1}\left ( n-N \right )^{1-\frac{3}{2}\left ( \frac{1}{p}-\frac{1}{r} \right )-\frac{s}{2}}.
	\end{displaymath}
 Note that $ C\left( n\right),\tilde{C}\left( n\right)\to 0 $ as $ n\to \infty $. This completes the proof.
	
\end{proof}

\begin{theorem}
	Under assumptions of Lemma $\ref{lemma 3.1}$, there exist positive constants $\delta$ such that for every $ f\in BC\left ( \mathbb{R};{\dot{B}_{p,2} ^{-s} } \left ( \mathbb{R}^{3} \right ) \right )\cap BC\left ( \mathbb{R};{L^{\kappa}} \left ( \mathbb{R}^{3} \right ) \right )  $ satisfying $ \left \| f \right \|_{\mathcal{L} } \le \delta $. If $ \alpha = 0 $, i.e. $ Q = \begin{pmatrix}
		\cos \theta  & -\sin\theta   & 0 \\
		\sin \theta & \cos \theta & 0 \\
		0 & 0 & 1
	\end{pmatrix} $, there exists a positive constant $ K $ such that we have the following:\\
	(i) If $ f $ is $ \left( Q,T\right) $-affine-periodic, $ \eqref{1.1} $ admits a unique rotating periodic solution $ u\in X_{K}^{r,q} $ near the equilibrium; \\
	(ii) If $ f $ is spiral-like almost periodic, $ \eqref{1.1} $ admits a unique spiral-like almost periodic mild solution $ u\in X_{K}^{r,q} $ near the equilibrium;\\
	(iii) If $ f $ is spiral-like almost automorphic, $ \eqref{1.1} $ admits a unique spiral-like almost automorphic mild solution $ u\in X_{K}^{r,q} $ near the equilibrium,\\ 
    where $ X_{K}^{r,q} := \left \{u\in BC\left ( \mathbb{R};L_{\sigma }^{r} \left ( \mathbb{R}^{3} \right ) \right )\cap BC\left ( \mathbb{R}; \dot{W}^{1,q}\left ( \mathbb{R}^{3} \right ) \right ) \mid \left \| u \right \|_{X^{r,q} } \le K \right \}  $.
\end{theorem}
\begin{proof}
	(i) Since $ f $ is $ \left( Q,T\right) $-affine-periodic, we first show that equation $ \eqref{1.2} $ has a unique mild solution. To simplify arguments, by scaling $ v\to \delta v $,$ f\to \delta f $, $ \eqref{2.6} $ becomes
	\begin{displaymath}
		v\left ( t \right ):=-\int_{-\infty }^{t} \delta T \left ( t-\tau  \right ) \mathbb{P} \left [ \left ( v\left ( \tau  \right ) \cdot \nabla  \right )v\left ( \tau  \right )   \right ]d\tau + \int_{-\infty }^{t}  T \left ( t-\tau  \right ) \mathbb{P}   f\left ( \tau  \right )  d\tau.
	\end{displaymath}
	By Lemma \ref{lemma 3.1}, we have the following estimate:
	\begin{equation}
		\begin{split}\label{3.15}
			\left \| v\left ( t \right )  \right \| _{X^{r,q}}&\le \delta C\left ( n \right ) \left \{ \left ( \sup_{t\in \mathbb{R}}\left \| v\left ( t \right )  \right \| _{r} \right )^{2} + \left ( \sup_{t\in \mathbb{R}}\left \| \nabla v\left ( t \right )  \right \| _{q} \right )^{2}\right \} \\
			&\quad + \delta C\left \{ \left ( \sup_{\left | t \right |\le n }\left \| v\left ( t \right )  \right \|_{r}   \right ) ^{2} +  \left ( \sup_{\left | t \right |\le n}\left \| \nabla v\left ( t \right )  \right \|_{q}   \right ) ^{2}\right \} \\
			&\quad + \tilde{C}\left ( n \right ) \left \| f \right \| _{\mathcal{L} } + \tilde{C} \left \| f \right \| _{\mathcal{K}_{N} }.
		\end{split}  
	\end{equation}
	Let 
	\begin{displaymath}
		\varPhi\left ( u,v,f \right )\left( t\right) :=-\int_{-\infty }^{t} T \left ( t-\tau  \right ) \mathbb{P} \left [ \left(  u\left ( \tau  \right ) \cdot \nabla \right) v\left ( \tau  \right ) \right ]d\tau + \int_{-\infty }^{t}  T \left ( t-\tau  \right ) \mathbb{P}   f\left ( \tau  \right )  d\tau.
	\end{displaymath}
	\indent\setlength{\parindent}{2em} Step 1: We show that $ \Phi $ maps $ X^{r,q}_{K} $ into itself. Fix $ N=1,n=2 $, by $ \eqref{3.15} $, we have
	\begin{displaymath}
		\begin{split}
			\left \| v \right \| _{X^{r,q}}&\le C_{1}\left \|v  \right \| _{X^{r,q}}^{2} + C_{2}\left ( \sup\limits_{\mathbb{R} } \left \| f\left ( \cdot  \right )  \right \|_{\dot{B}_{p,2} ^{-s}} + {\sup\limits_{\mathbb{R} } \left \| f\left ( \cdot  \right )  \right \|_{\kappa }}  \right ) \\
			&= C_{1}\left \|v \right \| _{X^{r,q}}^{2} + C_{2},
		\end{split}
	\end{displaymath}
	where $ C_{1}=\delta C\left( 2\right) +\delta C $,$ C_{2}=\tilde{C}\left( 2\right) +\tilde{C} $. Let $ K $ be the smaller root of the equation $ C_{1}x^{2}-x+C_{2} =0 $,
	\begin{displaymath}
		K=\frac{1-\sqrt{1-4C_{1}C_{2}} }{2C_{1}} = \frac{2C_{2} }{1+\sqrt{1-4C_{1}C_{2}}} \to C_{2}\quad as \quad \delta \to 0^{+}.
	\end{displaymath}
	Hence when $ \left \| v \right \| _{X^{r,q}} < K $, we have 
	\begin{displaymath}
		\begin{split}
			\left \| \Phi \left ( v,v,f \right )  \right \| _{X^{r,q}} &< C_{1} \left ( \frac{1-\sqrt{1-4C_{1}C_{2}} }{2C_{1}} \right ) ^{2}+C_{2}\\
			&=K.
		\end{split}
	\end{displaymath}
	Then the operator $ \Phi $ maps $ X^{r,q}_{K} $ into itself. \\
	\indent\setlength{\parindent}{2em} Step 2: We will show that $ \Phi $ is a contraction mapping on $ X^{r,q}_{K} $. Suppose that $ v\left( t,x\right)  $ and $ \hat{v}\left( t,x\right)  $ are bounded mild solutions to $ \eqref{1.2} $. Then, we have
	\begin{displaymath}
		\Phi \left [ v,v,f \right ] - \Phi \left [ \hat{v},\hat{v},f \right ] = \Phi \left [ v,v-\hat{v},0 \right ] + \Phi \left [ v-\hat{v},\hat{v},0 \right ] .
	\end{displaymath}
	It follows that 
	\begin{displaymath}
		\begin{split}
			\left \| \Phi \left [ v,v,f \right ] - \Phi \left [ \hat{v},\hat{v},f \right ]  \right \|_{X^{r,q}} 
			\le  C_{1}\left( \left \| v \right \| _{X^{r,q}} + \left \| \hat{v} \right \| _{X^{r,q}}\right) \left \| v-\hat{v} \right \| _{X^{r,q}} 
			\le 2C_{1}K \left \| v-\hat{v} \right \| _{X^{r,q}}.
		\end{split}
	\end{displaymath}
    There are constants $ \delta $ satisfing $ 0<\delta<1 $ such that
	\begin{displaymath}
		2C_{1}K = 1-\sqrt{1-4C_{1}C_{2}} < 1.
	\end{displaymath}
	Hence, $ \Phi \left ( v,v,f \right ) $ is a contraction mapping on $ X^{r,q}_{K} $. According to Banach's theorem, there exists a unique bounded mild solution $ v\left( t,x\right) $ to equation $ \eqref{1.2} $. Specifically, when $ \alpha = 0 $, equation $ \eqref{1.2} $ is the original equation $ \eqref{1.1} $. Thus, for equation $ \eqref{1.1} $, there exists a unique bounded mild solution $ u \left( t,x\right) $. In the following, we only need to show that $ Q^{-1} u\left ( t+T,Qx \right ) $ is a solution to equation $ \eqref{1.1} $, where $ Q = \begin{pmatrix}
		\cos \theta  & -\sin\theta   & 0 \\
		\sin \theta & \cos \theta & 0 \\
		0 & 0 & 1
	\end{pmatrix} $ is the orthogonal matrix.\\
	Since $ u\left( t,x\right)  $ is a mild solution of equation $ \eqref{1.1} $, it satisfies
	\begin{displaymath}
		\frac{\partial u\left( t,x\right) }{\partial t} = \Delta u\left( t,x\right)  - le_{3}\times u\left( t,x\right)  - \left ( u\left( t,x\right) \cdot \nabla  \right ) u\left( t,x\right)  - \nabla \pi\left( t,x\right)   + f\left( t,x\right).
	\end{displaymath}
	Then, we have
	\begin{equation*}
		\begin{aligned}
			&\frac{\partial Q^{-1} u\left ( t+T,Qx \right )}{\partial t} = Q^{-1} \frac{\partial u\left ( t+T,Qx \right )}{\partial \left ( t+T \right )} = Q^{-1} \frac{\partial u\left ( t+T,y \right )}{\partial \left ( t+T \right )}\\
			&=  Q^{-1} [ {\Delta  _{y} u\left ( t+T,y \right )}  - { le_{3}\times u\left ( t+T,y \right )} - {\left ( u\left ( t+T,y \right) \cdot \nabla _{y} \right )u\left ( t+T,y \right)} \\ & \quad -  {\nabla _{y}\pi \left ( t+T,y \right)} + {f\left ( t+T,y \right)} ] \\
			&=  Q^{-1} [ Q {\Delta \left (Q^{-1}u\left ( t+T,Qx \right )  \right ) }  - {Q le_{3}\times \left ( Q^{-1}u\left ( t+T,Qx \right ) \right ) }\\&\quad - {Q\left (Q^{-1} u\left ( t+T,Qx \right) \cdot \nabla  \right )\left ( Q^{-1}u\left ( t+T,Qx \right) \right ) } - {Q\nabla \pi \left ( t,x \right)} \\&\quad + {f\left ( t+T,Qx \right)} ]\\
			&={\Delta \left (Q^{-1}u\left ( t+T,Qx \right )  \right ) }  - {le_{3}\times \left ( Q^{-1}u\left ( t+T,Qx \right ) \right ) }\\&\quad - {\left (Q^{-1} u\left ( t+T,Qx \right) \cdot \nabla  \right )\left ( Q^{-1}u\left ( t+T,Qx \right) \right ) } - {\nabla \pi \left ( t,x \right)}\\&\quad + {Q^{-1}f\left ( t+T,Qx \right)} \\
			&={\Delta \left (Q^{-1}u\left ( t+T,Qx \right )  \right ) }  - {le_{3}\times \left ( Q^{-1}u\left ( t+T,Qx \right ) \right ) }\\&\quad - {\left (Q^{-1} u\left ( t+T,Qx \right) \cdot \nabla  \right )\left ( Q^{-1}u\left ( t+T,Qx \right) \right ) } - {\nabla \pi \left ( t,x \right)} + {f\left ( t,x \right)},
		\end{aligned}
	\end{equation*}
	where\begin{equation*}
	y := Qx = \begin{pmatrix}
			\cos \theta  & -\sin\theta   & 0 \\
			\sin \theta & \cos \theta & 0 \\
			0 & 0 & 1
		\end{pmatrix} \begin{pmatrix}
			x_{1} \\
			x_{2} \\
			x_{3}
		\end{pmatrix}  = \begin{pmatrix}
			x_{1} \cos \theta  - x_{2} \sin \theta  \\
			x_{1} \sin \theta  + x_{2} \cos \theta  \\
			x_{3}
		\end{pmatrix}.  
	\end{equation*}
	On the other hand, by Lemma $ \ref{lemma 3.1} $, $ Q^{-1} u\left ( t+T,Qx \right ) $ is bounded,
	\begin{displaymath}
		\begin{split}
			\left \| Q^{-1} u\left ( t+T,Qx \right ) \right \| \le \left \| Q^{-1} \right\| \left\| u\left ( t+T,Qx \right ) \right \|   
			\le C_{1} \left \| u\left ( t \right )  \right \|^{2}_{X^{r,q}} + C_{2}.
		\end{split}
	\end{displaymath}
By the uniqueness of the mild solution, we have $ u\left( t,x\right) = Q^{-1} u\left ( t+T,Qx \right ) $.
	Thus, equation $\eqref{1.1}$ has a unique rotating periodic solution.\\
	(ii) Based on the previous derivation of the semigroup of operators of equation $ \eqref{1.2} $, we obtain a mild solution of $ \eqref{1.1} $:
	\begin{displaymath}
		\begin{split}
			u\left ( t \right ) :=-\int_{-\infty}^{t} H\left ( t-\tau   \right ) \mathbb{P} \left [ \left ( u\left ( \tau  \right ) \cdot \nabla \right )u\left ( \tau \right ) \right ]d\tau + \int_{-\infty}^{t} H\left ( t-\tau   \right ) \mathbb{P} f\left( \tau\right) d\tau,
		\end{split}
	\end{displaymath}
	where 
	\begin{displaymath}
		\begin{split}
			H\left( t\right)g := \mathcal{F}^{-1} \left[  \cos\left( l\frac{\xi_{3}}{\left| \xi\right|}t\right) e^{-\left| \xi\right|^{2}t}I\widehat{g}\left( \xi\right) + \sin\left( l\frac{\xi_{3}}{\left|\xi\right|}t \right) e^{-\left| \xi\right|^{2}t}R^{\ast}\left( \xi\right) \widehat{g}\left( \xi\right)\right],
		\end{split}
	\end{displaymath}
	\begin{displaymath}
		R^{\ast}\left ( \xi \right ) = \frac{1}{\left | \xi \right |}\begin{pmatrix}
			0 & \xi_{3} & -\xi_{2} \\
			-\xi_{3} & 0 & \xi_{1} \\
			\xi_{2} & -\xi_{1} & 0
		\end{pmatrix},t \ge 0,g \in L_{\sigma}^{p}\left( \mathbb{R} ^{3}\right).
	\end{displaymath}
	Similar to the proof above, we obtain that there is a unique mild solution $ u\in X_{K}^{r,q} $ to $ \eqref{1.1} $.
	For simplicity we write $ w\left( t,x\right) :=Qu\left(t,Q^{-1}x \right) $. Using the coordinate transformation, we obtain the following equation:
	\begin{equation}\label{3.16}
		\frac{\partial w\left ( t,x \right ) }{\partial t} = \Delta w\left( t,x\right)  - le_{3}\times w\left( t,x\right)  - \left ( w\left( t,x\right) \cdot \nabla  \right ) w\left( t,x\right)  - \nabla \pi\left( t,Q^{-1}x\right) + Qf\left( t,Q^{-1}x\right),
	\end{equation}
	whose mild solution is
	\begin{displaymath}
		w\left ( t \right ) :=-\int_{-\infty}^{t} H\left ( t-\tau   \right ) \mathbb{P} \left [ \left ( w\left ( \tau  \right ) \cdot \nabla \right )w\left ( \tau \right ) \right ]d\tau + \int_{-\infty}^{t} H\left ( t-\tau   \right ) \mathbb{P}\left[  Qf\left( \tau,Q^{-1}x\right) \right]d\tau.
	\end{displaymath}
 Since $ f $ is spiral-like almost periodic, for every $ \varepsilon >0 $ there exists a $ T\left( \varepsilon\right) >0 $ such that for any interval $ I $ of length $ T $ there is an $ s $ in $ I $ such that $ \sup\limits_{\mathbb{R}}\left \| f\left( s+\cdot ,x\right) -Qf\left( \cdot,Q^{-1}x\right) \right \| _{\mathcal{L}} <\varepsilon $.\\
	From variable substitution, we have
	\begin{displaymath}
		\begin{split}
			u\left( s+t,x\right) &= \int_{-\infty}^{s+t} H\left ( s+t-\tau   \right ) \left\lbrace -\mathbb{P} \left[ \left ( u\left ( \tau  \right ) \cdot \nabla \right )u\left ( \tau \right ) \right]+\mathbb{P} f\left( \tau\right)\right\rbrace d\tau \\
			&= \int_{-\infty}^{t} H\left ( t-\tau \right ) \left\lbrace -\mathbb{P} \left[ \left ( u\left ( s+\tau \right ) \cdot \nabla \right )u\left ( s+\tau \right ) \right]+\mathbb{P} f\left( s+\tau\right)\right\rbrace d\tau.
		\end{split}
	\end{displaymath}
Similar to the proof of Lemma $ \ref{lemma 3.1} $, we have
	\begin{displaymath}
		\begin{split}
		&\left \| u\left ( s+t,x \right ) -Qu\left ( t,Q^{-1}x \right )  \right \|_{X^{r,q}} =\left \| u\left ( s+t,x \right ) -w\left ( t,x \right )  \right \|_{X^{r,q}}\\
		&= \left\| -\int_{-\infty}^{t} H\left ( t-\tau \right )\left\lbrace  \mathbb{P} \left[\left ( u\left ( \tau+s  \right ) \cdot \nabla \right )u\left ( \tau+s \right )\right] -\mathbb{P}\left[ \left ( w\left ( \tau \right ) \cdot \nabla \right )w\left ( \tau\right ) \right]\right\rbrace d\tau\right.\\
		&\left.\quad + \int_{-\infty}^{t} H\left ( t-\tau \right )\left\lbrace  \mathbb{P} f\left( \tau+s\right)-\mathbb{P} \left[ Qf\left( \tau,Q^{-1}x\right)\right] \right\rbrace  d\tau\right\| _{X^{r,q}}\\
		& \le C_{1}\left( \left \| u\left( s+t,x\right)  \right \| _{X^{r,q}} + \left \| w\left( t,x\right)  \right \| _{X^{r,q}}\right) \left \| u\left ( s+t,x \right ) -w\left ( t,x \right )  \right \|_{X^{r,q}} \\
		&\quad+ C_{2}\sup\limits_{\mathbb{R}}\left \| f\left( s+\cdot ,x\right) -Qf\left( \cdot,Q^{-1}x\right) \right \| _{\mathcal{L}}\\
		& \le 2C_{1}K\left \| u\left ( s+t,x \right ) -w\left ( t,x \right )  \right \|_{X^{r,q}} +C_{2}\varepsilon.
		\end{split}
	\end{displaymath}
Hence,
\begin{displaymath}
	\left \| u\left ( s+t,x \right ) -Qu\left ( t,Q^{-1}x \right )  \right \|_{X^{r,q}} \le \frac{C_{2}\varepsilon}{1-2C_{1}K},
\end{displaymath}
where $ 2C_{1}K = 1-\sqrt{1-4C_{1}C_{2}} < 1 $. Then $ u\left( t, x\right) $ is the spiral-like almost periodic mild solution.\\
(iii) Since $ f $ is spiral-like almost automorphic, for each sequence $ \left\lbrace t'_{n}\right\rbrace $ there exists a subsequence $ \left\lbrace t_{n}\right\rbrace  $ and $ \tilde{f}\in BC\left ( \mathbb{R};{\dot{B}_{p,2} ^{-s} } \left ( \mathbb{R}^{3} \right ) \right )\cap BC\left ( \mathbb{R};{L^{\kappa}} \left ( \mathbb{R}^{3} \right ) \right ) $ such that 
\begin{displaymath}
	\begin{split}
		\lim\limits_{n \to \infty} \left \| f\left( t+t_{n} ,x\right) -Q\tilde{f}\left( t,Q^{-1}x\right) \right \| _{\mathcal{L}} = 0 , \\
		\lim\limits_{n \to \infty} \left \| \tilde{f}\left( t-t_{n} ,x\right) -Qf\left( t,Q^{-1}x\right) \right \| _{\mathcal{L}} = 0 
	\end{split}
\end{displaymath}
for each $ t\in \mathbb{R} $. By the previous proof, there is $ \tilde{u}\left ( t \right ) $ with $ \left \| \tilde{u}\left ( t \right )  \right \| _{X^{r,q} }\le K $ satisfying the integral equation
\begin{displaymath}
	\tilde{u}\left ( t,x \right ) := -\int_{-\infty}^{t} H\left ( t-\tau   \right ) \mathbb{P} \left [ \left ( \tilde{u}\left ( \tau  \right ) \cdot \nabla \right )\tilde{u}\left ( \tau \right ) \right ]d\tau + \int_{-\infty}^{t} H\left ( t-\tau \right ) \mathbb{P} \tilde{f}\left( \tau,x\right) d\tau.
\end{displaymath}
Similarly, $ \tilde{w}\left( t,x\right) := Q\tilde{u}\left ( t,Q^{-1}x \right )$, then $ \left \| \tilde{w}\left ( t \right )  \right \| _{X^{r,q} }\le K $ and $ \tilde{w}\left( t,x\right) $ satisfies the following equation:
\begin{displaymath}
	\tilde{w}\left ( t,x \right ) = -\int_{-\infty}^{t} H\left ( t-\tau   \right ) \mathbb{P} \left [ \left ( \tilde{w}\left ( \tau  \right ) \cdot \nabla \right )\tilde{w}\left ( \tau \right ) \right ]d\tau + \int_{-\infty}^{t} H\left ( t-\tau   \right ) \mathbb{P}\left[  Q\tilde{f}\left( \tau,Q^{-1}x\right) \right]d\tau.
\end{displaymath}
	From variable substitution, we have
\begin{displaymath}
	\begin{split}
		u\left( t+t_{n},x\right) &= \int_{-\infty}^{t+t_{n}} H\left ( t+t_{n}-\tau   \right ) \left\lbrace -\mathbb{P} \left[ \left ( u\left ( \tau  \right ) \cdot \nabla \right )u\left ( \tau \right ) \right]+\mathbb{P} f\left( \tau\right)\right\rbrace d\tau \\
		&= \int_{-\infty}^{t} H\left ( t-\tau \right ) \left\lbrace -\mathbb{P} \left[ \left ( u\left ( \tau+t_{n}  \right ) \cdot \nabla \right )u\left ( \tau+t_{n} \right ) \right]+\mathbb{P} f\left( \tau+t_{n}\right)\right\rbrace d\tau,
	\end{split}
\end{displaymath}
and
\begin{displaymath}
	\begin{split}
	    &u\left ( t+t_{n},x \right ) -\tilde{w}\left ( t,x \right ) \\
		&= -\int_{-\infty}^{t} H\left ( t-\tau \right )\left\lbrace  \mathbb{P} \left[\left ( u\left ( \tau+t_{n}  \right ) \cdot \nabla \right )u\left ( \tau+t_{n} \right )\right] -\mathbb{P}\left[ \left ( \tilde{w}\left ( \tau \right ) \cdot \nabla \right )\tilde{w}\left ( \tau\right ) \right]\right\rbrace d\tau\\
		&\quad + \int_{-\infty}^{t} H\left ( t-\tau \right )\left\lbrace  \mathbb{P} f\left( \tau+t_{n}\right)-\mathbb{P} \left[ Q\tilde{f}\left( \tau,Q^{-1}x\right)\right] \right\rbrace  d\tau  \\
		&:=-G\left( u,\tilde{w}\right) \left( t\right) + P\left(t \right).
	\end{split}
\end{displaymath}
Without loss of generality, fix $ t\in \mathbb{R} $. For any given $ \varepsilon>0 $, there exists $ N > 0 $ such that
\begin{displaymath}
	\begin{split}
		&\left \| \int_{-\infty}^{t-N} H\left ( t-\tau \right )\left\lbrace  \mathbb{P} \left[\left ( u\left ( \tau+t_{n}  \right ) \cdot \nabla \right )u\left ( \tau+t_{n} \right )\right] -\mathbb{P}\left[ \left ( \tilde{w}\left ( \tau \right ) \cdot \nabla \right )\tilde{w}\left ( \tau\right ) \right]\right\rbrace d\tau \right \|_{X^{r,q}} \\
		&\quad + \left \| \int_{-\infty}^{t-N} H\left ( t-\tau \right )\left\lbrace  \mathbb{P} f\left( \tau+t_{n}\right)-\mathbb{P} \left[ Q\tilde{f}\left( \tau,Q^{-1}x\right)\right] \right\rbrace  d\tau \right \|_{X^{r,q}}\\ 
		&\le 2C\left( N\right) K^{2} + 2\tilde{C}\left( N\right) \sup\limits_{t\in \mathbb{R}}\left(\left \| f\left ( t+t_{n} \right )  \right \| _{\mathcal{L}} + \left \| Q\tilde{f}\left( t,Q^{-1}x\right) \right \| _{\mathcal{L}} \right) \\
		&\le \frac{\varepsilon }{3}. 
	\end{split}
\end{displaymath}
From Egoroff’s theorem, given any $ \eta > 0 $, there exists a closed set $ B\subset \left [ t-N,t \right ]  $, with measure $ \left | B \right | <\eta $ and a subsequence $ \left\lbrace t_{n}\right\rbrace  $ and an integer $ N_{1} > 0 $ such that
\begin{displaymath}
		\sup\limits_{n>N_{1}}\sup\limits_{\left [ t-N,t \right ]\setminus B}\left \| f\left ( \cdot+t_{n}\right ) -Q\tilde{f}\left( \cdot,Q^{-1}x\right)  \right \| _{\mathcal{L}} <\frac{\varepsilon }{3\left ( C_{2}+1 \right ) }.
\end{displaymath}
We have
\begin{displaymath}
	\begin{split}
		G\left( u,\tilde{w}\right) \left( t\right) 
		&=\int_{-\infty}^{t-N} H\left ( t-\tau \right )\left\lbrace  \mathbb{P} \left[\left ( u\left ( \tau+t_{n}  \right ) \cdot \nabla \right )u\left ( \tau+t_{n} \right )\right] -\mathbb{P}\left[ \left ( \tilde{w}\left ( \tau \right ) \cdot \nabla \right )\tilde{w}\left ( \tau\right ) \right]\right\rbrace d\tau\\
		&\quad + \! \int_{\left [ t-N,t \right ]\setminus B } \!\! H\left ( t-\tau \right )\left\lbrace  \mathbb{P} \left[\left ( u\left ( \tau+t_{n}  \right ) \cdot \nabla \right )u\left ( \tau+t_{n} \right )\right] -\mathbb{P}\left[ \left ( \tilde{w}\left ( \tau \right ) \cdot \nabla \right )\tilde{w}\left ( \tau\right ) \right]\right\rbrace d\tau\\
		&\quad +\! \int_{ B }H\left ( t-\tau \right )\left\lbrace  \mathbb{P} \left[\left ( u\left ( \tau+t_{n}  \right ) \cdot \nabla \right )u\left ( \tau+t_{n} \right )\right] -\mathbb{P}\left[ \left ( \tilde{w}\left ( \tau \right ) \cdot \nabla \right )\tilde{w}\left ( \tau\right ) \right]\right\rbrace d\tau\\
		&:=G_{1}\left( t\right) +G_{2}\left( t\right)+G_{3}\left( t\right),
		\end{split}
\end{displaymath}
and
\begin{displaymath}
	\begin{split}
		P\left( t\right)
	    &=\int_{-\infty}^{t-N} H\left ( t-\tau \right )\left\lbrace  \mathbb{P} f\left( \tau+t_{n}\right)-\mathbb{P} \left[ Q\tilde{f}\left( \tau,Q^{-1}x\right)\right] \right\rbrace  d\tau\\
		&\quad + \int_{\left [ t-N,t \right ]\setminus B }H\left ( t-\tau \right )\left\lbrace  \mathbb{P} f\left( \tau+t_{n}\right)-\mathbb{P} \left[ Q\tilde{f}\left( \tau,Q^{-1}x\right)\right] \right\rbrace d\tau\\
		&\quad +\int_{ B }H\left ( t-\tau \right )\left\lbrace  \mathbb{P} f\left( \tau+t_{n}\right)-\mathbb{P} \left[ Q\tilde{f}\left( \tau,Q^{-1}x\right)\right] \right\rbrace  d\tau\\
		&:=P_{1}\left( t\right)+P_{2}\left( t\right)+P_{3}\left( t\right).
	\end{split}
\end{displaymath}
We take $ \eta $ small enough such that
\begin{displaymath}
	\sup\limits_{B}\left\| G_{3}\left( t\right) +P_{3}\left( t\right) \right\| _{X^{r,q}} < \frac{\varepsilon }{3}.
\end{displaymath}
From the above inequalities, we obtain
     \begin{displaymath}
     	\begin{split}
     		\left \| u\left ( t+t_{n},x \right ) - \tilde{w}\left ( t,x \right )  \right \|_{X^{r,q}} 
     		&= \left \| -G\left( u,\tilde{w}\right) \left( t\right) + P\left(t \right) \right \|_{X^{r,q}}\\
     		&\le \frac{\varepsilon}{3} + \frac{\varepsilon}{3} + \frac{\varepsilon}{3} + \left \| G_{2} \left( t\right) \right \|_{X^{r,q}} \\
     		&\le \varepsilon +\left( 1-\sqrt{1-4C_{1}C_{2}} \right)\sup\limits_{\left [ t-N,t \right ]\setminus B}\left \| u\left ( \cdot+t_{n}\right ) -\tilde{w}\left( \cdot\right)  \right \| _{X^{r,q}},
     	\end{split}
     \end{displaymath}
 and thus
 \begin{displaymath}
 	\sup\limits_{\left [ t-N,t \right ]\setminus B}\left \| u\left ( \cdot+t_{n}\right ) -\tilde{w}\left( \cdot\right)  \right \| _{X^{r,q}} \le \frac{\varepsilon}{\sqrt{1-4C_{1}C_{2}}},
 \end{displaymath}
which implies
\begin{displaymath}
	\lim\limits_{n \to \infty}\left \| u\left ( t+t_{n},x\right ) -Q\tilde{u}\left( t,Q^{-1}x\right)  \right \| _{X^{r,q}} =0.
\end{displaymath}
Similarly, we have
\begin{displaymath}
	\lim\limits_{n \to \infty}\left \| \tilde{u}\left ( t-t_{n},x\right ) -Qu\left( t,Q^{-1}x\right)  \right \| _{X^{r,q}} =0.
\end{displaymath}
Then $ u\left( t, x\right)  $ is the spiral-like almost automorphic mild solution. This completes the proof.
\end{proof}

\bibliographystyle{plain}
\bibliography{ref} 

\end{document}